\documentclass[a4paper, leqno,12pt]{amsart}

\usepackage{amssymb, amsmath, amsthm,relsize}
\usepackage{xargs} 
\usepackage{hyperref} 
\usepackage[colorinlistoftodos,textsize=small]{todonotes} 
\usepackage{subcaption,graphicx,xcolor}
\usepackage{forest}

\numberwithin{equation}{section}
\newtheorem{thm}{Theorem}

\newtheorem{cor}[thm]{Corollary}
\newtheorem{lemma}[thm]{Lemma}

\theoremstyle{remark}
\newtheorem{ex}[thm]{Example}
\newtheorem*{theorem*}{Theorem}

\newtheorem{definition}{Definition}
 
         \newcommand{\p}{\mathfrak{p}}

        \newcommand{\McC}{\raise.5ex\hbox{c}}

\title[Rational inner functions, matrices, and singularities]{A note on polydegree $(n,1)$ rational inner functions, slice matrices, and singularities}

\author[Sola]{Alan Sola}
\address{Department of Mathematics, Stockholm University, Stockholm, S-106 91, Sweden}
\email{sola@math.su.se}

\date{\today}

\thanks{Partially supported by National Science Foundation Grant DMS
1928930 while participating in a spring 2022 program hosted by MSRI in Berkeley, CA}

\begin{document}

\begin{abstract} 
We analyze certain compositions of rational inner functions in the unit polydisk $\mathbb{D}^{d}$ with polydegree $(n,1)$, $n\in \mathbb{N}^{d-1}$, and isolated singularities in $\mathbb{T}^d$. Provided an irreducibility condition is met, such a composition is shown to be a rational inner function with singularities in precisely the same location as those of the initial function, and with quantitatively controlled properties. As an application, we answer a $d$-dimensional version of a question posed in \cite{BPS22} in the affirmative.
 \end{abstract}
\maketitle

\section{Introduction}
\subsection*{Background}
This note is concerned with certain bounded holomorphic functions on the unit polydisk in $\mathbb{C}^d$,  
\[\mathbb{D}^d=\{z=(z_1,\ldots, z_d)\in \mathbb{C}^d\colon |z_j|<1,\, \, j=1, \ldots, d\},\]
called rational inner functions, and their singularities on the $d$-torus 
\[\mathbb{T}^d=\{\zeta=(\zeta_1,\ldots, \zeta_d)\in \mathbb{C}^d\colon |\zeta_j|=1, \,\,j=1, \ldots, d\};\]
here and throughout, $d\in \mathbb{N}$. By Fatou's theorem for polydisks (see e.g. \cite[Chapter 3]{Rud69}), any bounded holomorphic function $\phi\colon \mathbb{D}^d\to \mathbb{C}$ has non-tangential boundary values 
$\phi^*(\zeta)=\angle \lim_{\mathbb{D}^d\ni z\to \zeta}\phi(z)$ at almost every point $\zeta \in \mathbb{T}^d$. If these boundary values satisfy $|\phi^*(\zeta)|=1$ for almost every $\zeta \in \mathbb{D}^d$, we say that $\phi$ is an inner function.

Inner functions of the form $\phi=q/p$, where $q,p \in \mathbb{C}[z_1,\ldots, z_d]$ and $p$ has no zeros in $\mathbb{D}^d$, are called rational inner functions (RIFs). In one variable, RIFs are precisely the finite Blaschke products in the unit disk $\mathbb{D}$. Blaschke products play a central role in function theory, see for instance \cite{GMR18} for an overview of the very rich theory of these functions. In two and more variables, RIFs are a concrete class of bounded holomorphic functions that is amenable to detailed study \cite[Chapter 5]{Rud69}, and appears naturally in several setting, for instance in connections with interpolation problems \cite{AM02}.

A classical result of Rudin and Stout (see \cite[Chapter 5]{Rud69}) states that any RIF in $\mathbb{D}^d$ admits a representation of the form 
\begin{equation}
\phi(z)=e^{ia}z^m\frac{\tilde{p}(z)}{p(z)},
\label{RSrep}
\end{equation}
where $a\in \mathbb{R}$, $m=(m_1,\ldots, m_d)\in \mathbb{N}^d$, and $\tilde{p}$ is the reflection of a polynomial $p$ with no zeros in $\mathbb{D}^d$ known as a stable polynomial. The reflection polynomial is defined as
\[\tilde{p}(z)=z_1^{n_1}\cdots z_d^{n_d}\bar{p}\left(\frac{1}{\bar{z}_1}, \ldots, \frac{1}{\bar{z}_d}\right).\]
The vector $(n_1, \ldots, n_d)$ is referred to as the polydegree of $p$; each $n_j=\deg_{z_j}(p)$ is the degree of $p$ in the variable $z_j$. In this note, we shall strip out monomial factors and consider RIFs $\phi=e^{ia}\tilde{p}/p$; this simplifies formulas and is not material for the problem we study.

RIFs as well as more general bounded rational functions in two or more variables have been considered by a number of authors in recent years, often in connection with stable polynomials, representation formulas, and operator-theoretic problems. We cannot give a full overview here, but a sampler of related work might include papers of Anderson, Dritschel, and Rovnyak \cite{ADR08}; Ball, Sadosky, and Vinnikov \cite{BSV05}; Knese \cite{Kne11a, Kne11b, Kne11c}, and Koll\'ar \cite{KolPrep}.

A series of recent papers with Bickel and Pascoe\cite{BPS18, BPS20, BPS22}; Bickel, Knese, and Pascoe \cite{BKPS}; and Tully-Doyle \cite{ST22} deal with aspects of RIF theory that are particular to dimensions $d\geq 2$. Namely, unlike in one dimension, RIFs in two or more variables can have singularities on the $d$-torus, arising at points $\zeta\in \mathbb{T}^d$ where $p(\zeta)=0$ and $\tilde{p}(\zeta)=0$ vanish without having common factors that cancel out. A $d$-dimensional example (see \cite[Section 5]{Kne11c} and \cite[Example 2.5]{BPS22}) is given by 
\begin{equation}
\phi_d(z)=\frac{\tilde{p}(z)}{p(z)}=\frac{d\prod_{k=1}^dz_k-\sum_{j\in J}z_{j_1}\cdots z_{j_{d-1}}}{d-\sum_{k=1}^d z_k} \quad (d\geq 2)
\label{faveRIF}
\end{equation}
which has a singularity at $(1,\ldots, 1)\in \mathbb{T}^d$. Here, $J=\{(j_1, \ldots, j_{d-1})\in \mathbb{N}^d \colon 1\leq j_1<j_1<\cdots<j_{d-1}\leq d\}$. 

One would like to describe RIF singularities in detail, and there are different ways of doing this. The papers \cite{BPS18,BPS20,BPS22}, as well as \cite{BKPS}, investigate for which $\p \geq 1$ the partial derivative of a RIF has $\frac{\partial \phi}{\partial z_d} \in L^{\p}(\mathbb{T}^d)$. Roughly speaking, the smaller the maximal $\p$ for which integrability holds, the stronger the singularity of $\phi$; for the example \eqref{faveRIF}, the maximal integrability index is $\p=\frac{1}{2}(d+1)$; see \cite{BPS22} and \cite{BKPS} for comprehensive discussions. The paper \cite{BPS18} and the work of Bergqvist \cite{Bprep} also consider other notions of derivative integrability corresponding to norms of Dirichlet type. 

\subsection*{Overview of results}
The purpose of this short note is to present some straight-forward observations regarding $d$-variable RIFs of polydegree $(n,1)$, $n=(n_1,\ldots, n_{d-1}) \in \mathbb{N}^{d-1}$, 
and their singularities. This restricted class of functions is often singled as more amenable to analysis, see for instance \cite{Kne11a, BPS22, BCSMich}. If $\hat{\zeta}=(\zeta_1, \ldots \zeta_{d-1})\in \mathbb{T}^{d-1}$ is kept fixed and $\phi=\tilde{p}/p$ is a RIF in $\mathbb{D}^d$, the resulting one-variable function $\phi_{\hat{\zeta}}(z_d)$ is either a M\"obius transformation mapping the unit disk onto itself, or else is a unimodular constant. By encoding this fact in a $2\times 2$ matrix-valued function of $\hat{\zeta}$, and expressing the determinant of this matrix in terms of $\hat{\zeta}$-polynomials extracted from $p$ and $\tilde{p}$, we are able to read off certain geometric characteristics of such $\phi$.

This allows us to exhibit $d$-variable RIFs with prescribed singularity types, and hence derivative integrability properties, while keeping the $z_d$-degree of the resulting functions equal to $1$. As a specific application, we are able to answer a stronger version of \cite[Question 3]{BPS22} in the affirmative.

\section{Preliminaries}
\subsection*{Polydegree $(n,1)$ RIFs and their singularities}
Let $p$ be an irreducible stable polynomial in $\mathbb{D}^d$, the latter meaning that $\mathcal{Z}(p)=\{z\in \mathbb{C}^d\colon p(z)=0\}$ does not intersect $\mathbb{D}^d$. We assume throughout that $p$ has polydegree $(n,1)$ where $n=(n_1,\ldots, n_d)\in \mathbb{N}^{d-1}$ and that $p$ is atoral, which in this context means that $p$ and $\tilde{p}$ share no common factor, see \cite[Section 1.2]{BPS22}. Then we can decompose $p$ as a sum 
\begin{equation}
p(z)=p_1(z_1,\ldots, z_{d-1})+z_d\,p_2(z_1,\ldots, z_{d-1})
\label{pdecomp}
\end{equation}
where $p_1(\hat{z})$ and $p_2(\hat{z})$ are in $\mathbb{C}[z_1,\ldots, z_{d-1}]$, and similarly
\begin{equation}
\tilde{p}(z)=\tilde{p}_2(\hat{z})+z_d\,\tilde{p}_1(\hat{z}), \quad \tilde{p}_1, \tilde{p}_2\in \mathbb{C}[z_1,\ldots, z_{d-1}].
\label{tpdecomp}
\end{equation} 
As we are interested in singular RIFs $\phi=\tilde{p}/p$, we assume there exists at least one $\zeta \in \mathbb{T}^d$ such that $p(\zeta)=0$. A result of Pascoe \cite[Corollary 1.7]{P17} shows that if we assume $p$ is irreducible, then any zero $p$ on $\mathbb{T}^d$ gives rise to a singularity of $\phi$. We restrict attention to the class of such $p$ for which we have the additional property that $\mathcal{Z}(p)\cap \mathbb{T}^d$ is finite; we call the corresponding $\phi=\tilde{p}/p$ finite-singularity RIFs. 

\begin{definition}
Suppose $\phi=\tilde{p}/p$ is a finite-singularity RIF in $\mathbb{D}^d$ with a singularity at $(1,\ldots, 1)\in \mathbb{T}^d$. We say $\p^*\geq 1$ is a local $z_d$-derivative integrability index of $\phi$ if
\[\p^*=\sup_{\p\geq 1}\left\{\p\colon \frac{\partial \phi}{\partial z_d}\in L^{\p}_{\mathrm{loc}}(\mathbb{T}^d)\right\},\]
where each $L^{\p}_{\mathrm{loc}}(\mathbb{T}^d)$ is a standard local Lebesgue space on the $d$-torus.

The global $z_d$-derivative integrability index of $\phi$ is the maximum of all the local $z_d$-derivative integrability indices of the finite-singularity RIF $\phi$.
\end{definition}
Because of the argument principle, $\frac{\partial \phi}{\partial z_d}$ is integrable for any RIF so the assumption that $\p\geq 1$ is justified; see \cite{BPS18, Bprep} for details. In a similar way, we can define $z_j$-derivative indices. To keep this note as elementary as possible, we focus on the $z_d$-derivative integrability index of a $(n,1)$ RIF.

It is not a straight-forward task to determine local $z_j$-derivative indices of a $d$-variable RIF, or their global counterparts. Two-dimensional RIFs are much better understood than their general $d$-dimensional counterparts: for instance, the $z_1$ and $z_2$-derivative indices of a RIF coincide when $d=2$, but this is false when $d\geq 3$, and their values are determined by a geometric characteristic of $p$ at its zeros. See \cite{BPS18} and \cite{BKPS} for comprehensive presentations of the two-variable theory. 

As explained in \cite{BPS22}, the $z_d$-derivative integrability of a polydegree $(n,1)$ RIF $\phi$ is controlled by the rate at which the zero set of $\tilde{p}$ approaches $\mathbb{T}^d$ from inside $\mathbb{D}^d$. To make this statement precise, we return to the one-variable function $\phi_{\hat{\zeta}}$ and note that the $L^{\p}$ norm of the derivative of a M\"obius transformation is proportional to the distance to $\mathbb{T}$ of the point $\psi^0\in \mathbb{D}$ for which $\phi_{\hat{\zeta}}(\psi^0)=0$; see \cite[Lemma 4.2]{BPS20}. Solving $\tilde{p}(\hat{\zeta}, \psi^0)=0$ yields $\psi^0(\hat{\zeta})=-\tilde{p}_2(\hat{\zeta})/\tilde{p}_1(\hat{\zeta})$, where $\tilde{p}_1, \tilde{p}_2$ are the polynomials from \eqref{tpdecomp}. 

Therefore, we set
\[\rho_{\phi}(\hat{\zeta})=1-|\psi^0(\hat{\zeta})|^2=\frac{|\tilde{p}_1(\hat{\zeta})|^2-|\tilde{p}_2(\hat{\zeta})|^2}{|\tilde{p}_1(\hat{\zeta})|^2}.\]
Note that since $\phi$ was assumed to be a finite-singularity RIF, the polynomial $\tilde{p}_1$ has no zeros in $\mathbb{T}^{d-1}$; otherwise $\mathcal{Z}(\tilde{p})\cap \mathbb{T}^d$ would contain a vertical line \cite[Section 3]{BPS22}, which is impossible since zeros of $\tilde{p}$ on $\mathbb{T}^d$ are also zeros of $p$. Hence the vanishing of $\rho_{\phi}$ near a singularity is determined by the vanishing of its numerator. As a consequence of this discussion and \cite[Theorem 2.1]{BPS22}, we obtain the following criterion.
\begin{thm}\label{intcrit}
Suppose $\phi$ is a finite-singularity RIF with polydegree $(n,1)$ and a singularity at $(1,\ldots, 1) \in \mathbb{T}^d$. Then
$\frac{\partial \phi}{\partial z_d}\in L^{\p}_{\mathrm{loc}}(\mathbb{T}^d)$ at $(1,\ldots, 1)$ if and only if
\[\int_{U}\left[|\tilde{p}_1(\hat{\xi})|^2-|\tilde{p}_2(\hat{\zeta})|^2\right]^{1-\p}dm(\hat{\zeta})<\infty,\]
where $U\subset \mathbb{T}^{d-1}$ is any sufficiently small open set in $\mathbb{T}^{d-1}$ containing $(1,\ldots, 1)$.
\end{thm}

\subsection*{Polydegree $(n,1)$ rational inner functions and $2\times 2$ matrices}
Suppose $\phi=\tilde{p}/p$ is a finite-singularity RIF of polydegree $(n,1)$, and consider, for $\hat{\zeta}\in \mathbb{T}^{d-1}$ fixed, the one-variable function 
\[\phi_{\hat{\zeta}}(z_d)=\phi(\hat{\zeta},z_d).\]
Then, $\phi_{\hat{\zeta}}(z_d)$ is a rational function in $\mathbb{D}$, which attains unimodular boundary values at every point $\zeta_d \in \mathbb{T}$ by a theorem of Knese \cite[Theorem C]{Kne15}. Hence $\phi_{\hat{\zeta}}$ is either a M\"obius transformation of the unit disk, or else $\phi_{\hat{\zeta}}(z_d)$ is constant, and equal to some element of $\mathbb{T}$. The former obtains generically, but the latter possibility certainly occurs on some exceptional sets, as can be checked by considering $\phi_d(1,\ldots, 1,\zeta_d)$, where $\phi_d$ is the function in \eqref{faveRIF}.

Guided by this discussion, we make the following definition.
\begin{definition}\label{matrixdef}
The slice matrix of $\phi$ is the function $M_{\phi}\colon \mathbb{T}^{d-1}\to M_{2,2}(\mathbb{C})$ given by
\[M_{\phi}(\hat{\zeta})=\left(\begin{array}{cc}\tilde{p}_1(\hat{\zeta}) & \tilde{p}_2(\hat{\zeta})\\ p_2(\hat{\zeta}) & p_1(\hat{\zeta}) \end{array}\right).\]
The slice determinant of $\phi$ is the function $P_{\phi}\colon \mathbb{T}^{d-1}\to \mathbb{C}$ given by 
\[P_{\phi}(\hat{\zeta})=\det M_{\phi}(\hat{\zeta}).\]
\end{definition}
Formally, the numerator and the denominator of $\phi_{\hat{\zeta}}(z_d)$ can be read off from $M_{\phi}(\hat{\zeta})(z_d,\,\, 1)^T$. The slice determinant allows us to detect singularities of $\phi$ as well as their finer properties.

\begin{lemma}
The function $\phi_{\hat{\xi}}$ is constant if and only if $P_{\phi}(\hat{\xi})=0$, and this happens if and only if $(\hat{\xi}, \eta)$ is a singularity of $\phi$ for some value of $\eta\in \mathbb{T}$.
Moreover, $\frac{\partial \phi}{\partial z_d}\in L^{\p}_{\mathrm{loc}}(\mathbb{T}^d)$ at $(\hat{\xi},\eta)$ if and only if $\int_{B_{\epsilon}(\hat{\xi})} |P_{\phi}(\hat{\zeta})|^{1-\p}dm(\hat{\zeta})<\infty$ for sufficiently small $\epsilon>0$. 
\end{lemma}
\begin{proof}
The first assertion is a direct consequence of the following well-known facts that, for $a,b,c,d$ complex, $\mathfrak{m}(z)=(az+b)/(cz+d)$ furnishes a non-trivial M\"obius transformation of the Riemann sphere if and only if $ad-bc\neq 0$; if $ad-bc=0$ then $\mathfrak{m}$ is constant. See \cite{JSbook} for a comprehensive treatment of M\"obius transformations and their connections with matrix groups.

The second assertion is a consequence of the results in \cite[Subsection 3.2]{BCSMich}, see in particular \cite[Lemma 3.3]{BCSMich}. 

The third assertion essentially amounts to a computation. Namely,
\[\det M_{\phi}(\hat{\zeta})=\tilde{p}_1(\hat{\zeta})p_1(\hat{\zeta})-\tilde{p}_2(\hat{\zeta})p_2(\hat{\zeta}).\]
Observing that $\zeta_j=1/\bar{\zeta}_j$, $j=1, \ldots, d-1$, and examining the definition of reflection polynomials, the expression on the right-hand side can be rewritten (in standard multi-index notation) as
\[\hat{\zeta}^n\bar{p}_1(\hat{\zeta})p_1(\hat{\zeta})-\hat{\zeta}^n\bar{p}_2(\hat{\zeta})p_2(\hat{\zeta})=\hat{\zeta}^n\left(|\tilde{p}_1(\hat{\zeta})|^2-|\tilde{p}_2(\hat{\zeta})|^2\right).\]
The result now follows after taking moduli and appealing to Theorem \ref{intcrit}.
\end{proof}

\section{Compositions and local properties of singularities}
Given an $(n,1)$ finite-singularity RIF, we define the following sequence of functions. See \cite{ST22} for a fuller study of dynamical properties of mappings, especially skew-products, whose components are RIFs.
\begin{definition}\label{iterdef}
Let $\phi=\tilde{p}/p$ be a finite-singularity RIF of polydegree $(n,1)$. Then $\phi^2\colon \mathbb{D}^d\to \mathbb{C}$ is defined as
\[\phi^2(z)=(\phi_{\hat{\zeta}}\circ \phi_{\hat{\zeta}})(z_d), \quad (\hat{z},z_d)\in \mathbb{D}^d.\]
For any $N\in \mathbb{N}$ with $N\geq 3$, $\phi^N$ is defined inductively as $\phi_{\hat{\zeta}}\circ \phi^{N-1}_{\hat{\zeta}}$.
\end{definition}
The functions $\phi^N$ are clearly rational and holomorphic in $\mathbb{D}^d$. As can be seen from \eqref{pdecomp} and \eqref{tpdecomp}, the $z_j$-degree of $\phi^N$ is at most $N\cdot n_j$ for $j=1,\ldots, d-1$, and $\deg_{z_d}(\phi^N)\leq 1$. One complication that may arise is that the numerator and the denominator of the composite function may initially share a common factor. We always assume any such factors present are cancelled, in which case we get a polydegree drop in $\phi^N$.

\begin{lemma}\label{lem:samesings}
Suppose $\phi^N=\tilde{p}_N/p_N$ is as in Definition \ref{iterdef} and does not experience a polydegree drop. Then $\phi^N$ is a finite-singularity RIF with the same singularities as $\phi$.
\end{lemma}
\begin{proof}
Since $\phi$ maps $\mathbb{T}^d$ onto $\mathbb{T}$, Knese's theorem implies that each $(\phi^N)^*$ is unimodular. Hence $\phi^N$ is inner. 

Next, for $\hat{\zeta}\in \mathbb{T}^{d-1}$, computing the slice matrix of $\phi^N_{\hat{\zeta}}$ amounts to taking the matrix power $M^N_{\phi}(\hat{\zeta})
=M_{\phi}(\hat{\zeta})\cdots M_{\phi}(\hat{\zeta})$, see \cite{JSbook}. The assumption that $\phi^N$ has full polydegree implies there are no common factors in the matrices that would be cancelled in $\phi^N$. Then, by multiplicativity of determinants, $\det M_{\phi}^N(\hat{\zeta})$ vanishes if and only if $\det M_{\phi}(\hat{\zeta})$ does. Thus, the $\hat{\zeta}$-coordinates of the singularities of $\phi^N$ are the same as those of $\phi$. Since $\phi^N$ has degree $1$ in $z_d$, and since $\phi$ has a singularity on the line $\{\hat{\zeta}\}\times \mathbb{T}$, each such $\hat{\zeta}$ determines a unique $\eta\in \mathbb{T}$ such that $(\hat{\zeta}, \eta)\in \mathbb{T}^d$ is a singularity of $\phi^N$. 
\end{proof}
The following example illustrates that common factors can be eliminated by rotating $\phi$ by a suitable factor $e^{ia}$, $a\in \mathbb{R}$, or in other words by replacing $\tilde{p}$ by $e^{ia}\tilde{p}$. Doing this only affects $P_{\phi}$ up to a unimodular factor. 
\begin{ex}
Consider $\phi=-(2z_1z_2-z_1-z_2)/(2-z_1-z_2)$. As is shown by induction in \cite[Example 1]{ST22}, each $\phi^N$, $N=1,2,3,\ldots$, has bidegree $(1,1)$.

Next, consider $\phi=(2z_1z_2-z_1-z_2)/(2-z_1-z_2)$. Then 
\[\psi=\phi^2=\frac{4z_1^2z_2-z_1^2-3z_1z_2-z_1+z_2}{4-3z_1-z_2-z_1z_2+z_1^2},\]
is a RIF that often features as a second example of a singular RIF on the bidisk; see for instance \cite{AMY12, BPS18, BPS20}. In particular, while $\frac{\partial \phi}{\partial z_2}\in L^{\p}(\mathbb{T}^2)$ if and only if $\p<3/2$, it was shown by direct computation in \cite[Example 2]{BPS18} that $\frac{\partial \psi}{\partial z_2} \in L^{\p}(\mathbb{T}^2)$ if and only if $\p<5/4$. We now give a conceptual explanation for this finding.
\end{ex}

\begin{thm}\label{thm:main}
Suppose $\phi=\tilde{p}/p$ is a finite-singularity RIF of polydegree $(n,1)$, with a singularity at $\vec{1}=(1,\ldots, 1)$. Suppose the local $z_d$-derivative integrability index of $\phi$ at $\vec{1}$ is equal to $\p^*=1+\mathfrak{q}^*$, where $\mathfrak{q}^*\geq 0$. 

If $N\in \mathbb{N}$ and $\phi^N$ has full polydegree, then the RIF $\phi^N=\tilde{p}_N/p_N$ has local $z_d$-derivative integrability index equal to $1+\mathfrak{q}^*/N$ near $\vec{1}$.
\end{thm}
\begin{proof}
By Lemma \ref{lem:samesings}, $\phi^N$ is a RIF with the same singularities as $\phi$, and in particular $\phi^N$ has a singularity at $\vec{1}$.
Since $\tilde{p}_N$ and $p_N$ have no common factors that can be cancelled, the slice matrix of $\phi^N$ is equal to $M_{\phi}^N$, the $N$-fold power of the slice matrix of $\phi$. Hence the order of vanishing of the slice determinant of $\phi^N$ is equal to $N$ times the order of vanishing of the slice determinant of $\phi$. In other words, $\frac{\partial \phi}{\partial z_d}\in L^{\p}(\mathbb{T}^d)$ precisely when 
\[\int_{U}(|\tilde{p}_1^2(\hat{\zeta})|^2-|\tilde{p}_2(\hat{\zeta})|^2)^{N(1-\p)}dm(\hat{\zeta})\]
is finite for $U\supset \vec{1}$ sufficiently small. By our assumption on $\phi$, this holds if $N(1-\p)> -\mathfrak{q}^*$ and fails when $N(1-\p)<-\mathfrak{q}^*$, and the result follows.
\end{proof}
When $d=2$ and $\phi$ has a singularity at $(e^{i\eta_1},e^{i\eta_2})$, it can be shown that $1-|\psi^0(e^{i\theta_1})|^2\asymp (\theta_1-\eta_1)^{2K}$ for some $K\in \mathbb{N}$. The number $2K$ is called the $z_2$-contact order of $\phi$ at $(e^{i\eta_1},e^{i\eta_2})$; see \cite{BPS18,BPS20} for definitions and proofs. The assumption that $\phi$ has finitely many singularities becomes superfluous in two variables by B\'ezout's theorem, and we obtain the following.
\begin{cor}\label{thm:2main}
Suppose $\phi=\tilde{p}/p$ is a bidegree $(n_1,1)$ RIF in $\mathbb{D}^2$ with $s$ singularities having associated contact orders $\{2K_1, \ldots, 2K_s\}$, and suppose $\phi^N=\tilde{p}_N/p_N$ has full polydegree.

Then $\phi^N$ has $s$ singularities with contact orders $\{2NK_1, \ldots, 2NK_s\}$.
\end{cor}

\section{Applications}
\subsection*{Finding extraneous zeros of two-variable RIF denominators}
In \cite{P18}, Pascoe presents a way of constructing two-variable RIFs with at least one singularity where the local contact order can be prescribed to take any value $2n$, $n\in \mathbb{N}$. (Strictly speaking, the construction is given in the setting of the bi-upper half-plane, but it can readily be transferred to the bidisk by means of conjugation by a suitable M\"obius map. See \cite[Section 7]{BPS20}.) In particular, any positive even integer is the contact order of some RIF in $\mathbb{D}^2$. 

However, Pascoe's construction may produce additional singularities in $\phi$ and, to the author's knowledge, does not appear give any immediate information about their location or nature. In principle, this can be addressed by finding all zeros of the two-variable denominator $p$, and then using the techniques in \cite{BPS20, BKPS} to determine the associated contact orders. By examining the matrix-valued function $\zeta_1\mapsto M_{\phi}(\zeta_1)$ we can detect any such extraneous singularities and determine their contact orders in a fairly simple way. First, we compute $P_{\phi}(\zeta_1)=\det M_{\phi}(\zeta_1)$ and find the zeros $\{\zeta_1^1, \ldots \zeta_1^s\}$ of the one-variable polynomial $P_{\phi}$ that are located on the unit circle. Plugging these values into $p$, we find the point $\zeta_2\in \mathbb{T}$ at which the polynomial $p(\zeta_1^j, z_2)$ vanishes as a function of $z_2$. Finally, the order of vanishing of $P_{\phi}$ gives us the $z_2$-contact order of $\phi$ at each singularity. By \cite[Section 4]{BPS20}, this is equal to the $z_1$-contact order of $\phi$ as well, allowing us to read off the derivative integrability of $\phi$ at each singularity.

\begin{ex}[Example 7.4 of \cite{BPS20}]
Consider the two-variable RIF 
\[\phi(z_1,z_2)=\frac{4z_1^3z_2+z_1^3-z_1^2+3z_1+1}{4+z_2-z_1z_2+3z_1^2z_2+z_1^3z_2},\]
which is obtained using Pascoe's method.
His construction guarantees that $\phi$ has a singularity at $(-1,-1)$ with contact order equal to $4$.
The slice matrix associated with $\phi$ is
\[M_{\phi}(\zeta_1)=\left(\begin{array}{cc} 4\zeta_1^3 & \zeta_1^3-\zeta_1^2+3\zeta_1+1\\ \zeta_1^3+3\zeta_1^2-\zeta_1+1 &4\end{array}\right)\]
and has determinant
\[P_{\phi}(\zeta_1)=\det M_{\phi}(\zeta_1)=-(\zeta_1-1)^2(\zeta_1+1)^4.\]
We immediately discern that $\phi$ has an additional singularity at $(1,-1)$, with contact order equal to $2$, as was checked in an ad hoc way in \cite{BPS20}.
\end{ex}

\subsection*{Further derivative integrability cutoffs of $d$-variable RIFs}
In \cite{BPS22}, a glueing construction from \cite[Section 7]{BPS20} was adapted to three variables and was used to exhibit a three-variable RIF with a single isolated singularity and worse derivative integrability properties than the three-variable instance of \eqref{faveRIF}. The drawback of that example is that the RIF so constructed has tridegree $(2,2,2)$, which in turn causes the verification of its claimed derivative integrability cutoff to involve lengthy computations. Thus, \cite[Question 3]{BPS22} asked whether there exist tridegree $(n_1,n_2,1)$ RIFs manifesting the same phenomenon. The example below answers this in the affirmative, in all dimensions $d=3,4,5,\ldots$.
\begin{ex}\label{dpowerex}
For $d\geq 2$ fixed and $N\in \mathbb{N}$, we consider the RIF in \eqref{faveRIF} and its associated compositions RIFs $\phi^N_{d}=\tilde{p}_{d,N}/p_{d,N}$, all of degree $1$ in $z_d$. Reading off the slice matrix of $\phi_d$ from \eqref{faveRIF}, we check that $M_{\phi_d}(\vec{1})^N$ has non-zero entries. This means there are no common factors vanishing at $\vec{1}$ to cancel in $\phi_d^N$, and we can proceed as in Theorem \ref{thm:main}.

In \cite[Example 2.5]{BPS22}, it was shown that near $(1,\ldots, 1)\in \mathbb{T}^{d-1}$,
\[1-|\psi^0(e^{i\theta_1}, \ldots, e^{i\theta_{d-1}})|^2\asymp \sum_{k=1}^{d-1}\theta_k^2,\]
and hence 
\[\left[|\tilde{p}_1(e^{i\theta_1}, \ldots, e^{i\theta_{d-1}})|^2-|\tilde{p}_2(e^{i\theta_1}, \ldots, e^{i\theta_{d-1}})|^2\right]^N\asymp \left(\sum_{k=1}^{d-1}\theta_k^2\right)^N.\]
Then, $\frac{\partial \phi_{d,N}}{\partial z_d}\in L^{\p}(\mathbb{T}^d)$ if and only if 
\[\int_{B_{\epsilon}(\vec{0})}\left(\sum_{k=1}^{d-1}\theta_k^2\right)^{N(1-\p)}dm\]
converges for small $\epsilon>0$. In $d-1$-dimensional polar coordinates, this corresponds to the convergence of
\[\int_0^{1}r^{2N(1-\p)}r^{d-2}dr=\int_0^1r^{2N+d-2-2N\p}dr.\]
This integral is finite if and only if $-1<2N+d-2-2N\p$, and hence we obtain the $z_d$-derivative integrability cutoffs
\begin{equation}
\p^*(d,N)=1+\frac{d-1}{2N}.
\label{dNcutoffs}
\end{equation}
This extends the work in \cite[Example 2.5]{BPS22}, where it was shown that $\p^*(d,1)=1+(d-1)/2=(d+1)/2$. Next, since $\p^*(2,N)=1+1/(2N)$, we observe that any contact order is realized by a RIF with a unique singularity at $(1,1)\in \mathbb{T}^2$. 

Finally, the fact that $\p^*(3,2)=3/2$ shows that $\phi_{3,2}$, with denominator
\[p_{3,2}(z)=9-6z_1-6z_2-3z_3+z_1^2 +z_2^2 + 3 z_1z_2 + 2 z_1 z_3 + 2 z_2z_3 -3 z_1z_2z_3 ,\]
provides an example of a RIF providing a positive answer to \cite[Question 3]{BPS22}.
\end{ex}
\section*{Acknowledgments}
I am grateful to my coauthors on \cite{BKPS, ST22}, as well as J. Raissy and L. Vivas, for many interesting conversations. Thanks are also due to L. Bergqvist for helpful comments.
     
\end{document}